\documentclass{amsart}
\usepackage{amscd, amsfonts, amsmath, amsthm, amssymb}
\usepackage{tabularx}
\usepackage{ltablex}
\usepackage{longtable}
\usepackage{hhline}
\usepackage{srcltx}
\usepackage[
  backref,
  colorlinks = true
]{hyperref}

\newcommand{\Z}{\mathbb{Z}}
\newcommand{\F}{\mathbb{F}}
\newcommand{\Q}{\mathbb{Q}}

\providecommand{\set}[1]{ \left\{ #1 \right\} }

\theoremstyle{plain}
\newtheorem{theorem}{Theorem}[section]

\newtheorem{proposition}[theorem]{Proposition}
\newtheorem{corollary}[theorem]{Corollary}

\theoremstyle{definition}
\newtheorem{definition}[theorem]{Definition}
\newtheorem{example}[theorem]{Example}
\newtheorem{examples}[theorem]{Examples}

\theoremstyle{remark}
\newtheorem{remark}[theorem]{Remark}

\numberwithin{equation}{theorem}
\numberwithin{table}{section}

\begin{document}

\author{Jos\'e Alejandro {Lara Rodr\'iguez}}
\title{Relations between  multizeta values in characteristic $p$}
\address{
Facultad de Matem\'aticas de la Universidad Aut\'onoma de Yucat\'an
Anillo Perif\'erico Norte, Tablaje Cat. 13615, Colonia Chuburn\'a Hidalgo Inn,
M\'erida Yucat\'an, M\'exico
}
\address{Departamento de Control Autom\'atico del
Centro de Investigaci\'on y de Estudios Avanzados del IPN,
Av. Instituto Polit\'ecnico Nacional 2508, San Pedro Zacatenco, 07360,
M\'exico, D.F.}

\email{lrodri@uady.mx, jlara@ctrl.cinvestav.mx}
\date{April 19, 2011}

\begin{abstract}
We study relations between the multizeta values for function fields 
introduced by D. Thakur. 
The product $\zeta(a)\zeta(b)$  is a linear combination of multizeta values. 
For $q=2$, a full conjectural description of how
the product of two zeta values can be described as the sum of multizetas was
given by Thakur.
The recursion part of this recipe was generalized by the author.
In this paper, the main conjecture formulated by the author, as well 
as some conjectures of Thakur  are proved.
Moreover, for general $q$, we prove  closed formulas as well as a recursive 
recipe to express $\zeta(a)\zeta(b)$ as a sum of multizeta values. 
\end{abstract}

\maketitle

\section{Introduction}
 We refer to \cite{Waldschmidt}, and references in there, for a survey of many
exciting recent developments related to the  multizeta values introduced by
Euler. We refer to \cite{Goss96, Thakur} for discussion of Goss zeta 
functions  in connection with the function field arithmetic.

Dinesh Thakur \cite[Section 5.10]{Thakur} introduced two types of multizeta
values for function fields over finite fields of characteristic $p$, one
complex valued (generalizing the Artin-Weil zeta function) and the other with
values in Laurent series over finite fields (generalizing the Carlitz-Goss zeta
function). In this paper, we only focus on the latter. For its properties,
connections with Drinfeld modules and Anderson $t$-motives, we refer the reader
to \cite{A-T2, Thakur,Thakur_Multizeta08, Thakur_Shuffle09}.

Thakur  proves the existence of ``shuffle'' relations for the
multizeta values (for a general $A$ with a rational place at infinity)
\cite{Thakur_Shuffle09}. In particular, he shows that the product
of multizeta values can also be expressed as a sum of some multizeta values, so
that the $\F_p$-span of all multizeta values is an algebra. 
In the function field case, the identities are much more complicated than the
classical shuffle identities. In fact there are two types of identities, one
with $\F_p(t)$ coefficients and the other with $\F_p$ coefficients. 
(Note that although for many purposes a good analog of $\Q$ is $\F_q(t)$, 
the prime field in  characteristic $p$ ($0$ respectively) is $\F_p$ ($\Q$ 
respectively)). We concentrate only on the latter type.  

The results in \cite{Thakur_Shuffle09}, although effective, are not explicit
and bypass the explicit conjectures formulated in \cite{Thakur_Multizeta08,
Jalr, Jalr10}. 
In this paper, we use the ideas of the process in \cite{Thakur_Shuffle09} to
prove the main conjecture formulated in 
\cite{Jalr, Jalr10} and to give a closed formula for the expression 
of the product of zeta values as the sum of multizeta values
(Theorems~\ref{closed_formula_for_general_q} and
\ref{symmetric_closed_formula_for_general_q}). More
precisely, for general $q$, we prove that the expression of $\zeta(a)\zeta(b)$
as a sum of multizeta values is obtained by a recursive process; we effectively
determine the recursion length, the terms and the number of terms to be added at
each step of the recursion (Theorem~\ref{thm2412}). Also, when $q=2$, we give 
formulas to compute the initial values (Theorem~\ref{initial_values_q2} and
Corollary~\ref{symmetric_closed_formula_q2}) and,
using one of this formulas, some conjectures  in \cite{Thakur_Multizeta08} are
proved.

\section{Frequently used notation}

\begin{tabularx}{\linewidth}{l X}
$\Z$  & \{integers\}\\
$\Z_+$ &  \{positive integers\}\\
$q$ &  a power of a prime $p$, $q = p^s$\\
$\F_q$ & a finite field of $q$ elements\\
${t}$&  an independent variable\\
$A$ & the polynomial ring $\F_q[t]$\\
$A_+$ &   monics in  $A$\\
$K$ &   the function field $\F_q(t)$\\
$K_\infty$ & $= \F_q((1/t)) = $ the completion of $K$ at $\infty$\\
$A_d$ & $ \{\mbox{elements of }A \mbox{ of degree }d\}$\\
$A_{d^+}$ & $ A_d \cap A_+$ \\
$[n]$ &   $= t^{q^n}-t$\\
`even' & $  \mbox{multiple of } q-1$\\
$\deg$ & $ \mbox{function assigning to } a\in A \mbox{ its degree in } t$\\
$\ell(k)$ & sum of the digits of the base $q$ expansion of $k$\\
$\lfloor x \rfloor$ & the largest integer not greater than $x$\\
$\lceil x \rceil$ &  the smallest integer not less than $x$.
\end{tabularx}

\section{Thakur multizeta values}

For $s\in\Z_+$, the \emph{Carlitz zeta values} \cite{Goss96, Thakur} are
defined as
\begin{align*}
\zeta_A(s):=\sum_{a\in A_+} \frac{1}{a^s} \in K_\infty.
\end{align*}

 For $s\in \Z$ and $d\geq 0$, write
\begin{align*}
S_d(s) := \sum _{ a\in A_{d^+}} \frac{1}{a^{s}} \in K.
\end{align*}

Given integers $s_i\in \Z_+$ and $d\ge 0$ put
\begin{align*}
  S_d(s_1,\dotsc,s_r) = S_d(s_1) \sum _{d > d_2 >\dotsb > d_r \ge 0}
S_{d_2}(s_2) \dotsm S_{d_r}(s_r) \in K.
\end{align*}

For $s_i\in \Z_+$, Thakur defined the \emph{multizeta values}  \cite{Thakur,
Thakur_Multizeta08} by:
\begin{align*}
\zeta(s_1,\dotsc,s_r):= \sum _{d_1>\dotsb >d_r\ge 0} S_{d_1}(s_1) \dotsm
S_{d_r}(s_r)= \sum \frac{1}{a_1^{s_1}\dotsm a_r^{s_r} }\in K_\infty,
\end{align*}
where the second sum is over all $a_i\in A_+$ of degree $d_i$ such that
$d_1>\dotsb >d_r\ge 0$. We say that this multizeta value has depth $r$ and
weight  $\sum s_i$.

\section{Relations between multizeta values}

Recall that Euler's multizeta values  $\zeta$ (only in this paragraph, the
greek letter $\zeta$ will be
used to denote the
classical multizeta values) are defined by $\zeta(s_1, \dotsc, s_r) = \sum
\left(n_1^{s_1} \dotsm n_r^{s_r} \right)^{-1}$, where the sum is over positive
integers $n_1> n_2 >\dotsb >n_r$ and $s_i$ are positive integers, with $s_1>1$
(this condition is required for convergence). Since $n_1=n_2$, $n_1>n_2$ or $n_2>n_1$, 
we  have the ``sum shuffle
relation''
\begin{align*}
\zeta(a)\zeta(b)&= \sum _{n_1=1 } ^\infty \frac{1}{n_1^a} \sum _{n_2=1 } ^\infty
\frac{1}{n_2^b}
= \sum _{n_1=n_2}\frac{1}{n_1 ^{a+b}} + \sum _{n_1>n_2}\frac{1}{n_1^a n_2^b} +
\sum _{n_2>n_1}\frac{1}{n_2^b n_1^a }\\
&= \zeta(a+b)+\zeta(a,b)+\zeta(b,a).
\end{align*}

In the function field case, this sum shuffle relation fails because there are
many polynomials of a given degree. In contrast to the classical sum shuffle,
in the function field case the identities we get are much more involved.

For $s_1,s_2\in \Z_+$ put
\begin{align*}
S_d(s_1,s_2)=\sum _{\substack {d=d_1>d_2\\a_i\in
A_+}}\frac{1}{a_1^{s_1}a_2^{s_2}}
\end{align*}
where $d_i=\deg(a_i)$. For $a,b\in \Z_+$, we define
\begin{align*}
\Delta_d (a,b)=S_d(a)S_d(b)-S_d(a+b).
\end{align*}
We write $\Delta(a,b)$ for $\Delta_1(a,b)$. The definition implies
$\Delta_d(a,b)=\Delta_d(b,a)$.

The next two theorems (the second theorem in the reference has 
implications to higher genus function fields, but we state only 
a special case relevant to us) are  due to Thakur \cite[Theorems 1,2]{Thakur_Shuffle09}. 

\begin{theorem}\label{shuffle_thm1}
   Given $a,b \in \Z_+$, there are $f_i\in \F_p$ and $a_i\in
\Z_+$, so that
\begin{align}
\Delta_d(a,b) =  \sum f_i S_d(a_i,a+b-a_i) \label{Delta1}
\end{align}
holds for $d=1$.
\end{theorem}

\begin{theorem}\label{shuffle_thm2}
 Fix $A$.  If \eqref{Delta1} holds for some $f_i\in  \F_p^\times$ and distinct $a_i \in \Z_+$
for $d=1$, then \eqref{Delta1} holds for all $d\ge 0$. In this case, we have the shuffle
relation
$$ \zeta(a)\zeta(b) -\zeta(a+b) -\zeta(a,b) - \zeta(b,a)= \sum f_i
\zeta(a_i,a+b-a_i).$$
\end{theorem}

\section{Main results}

In this section we shall prove the recursion part of the main conjecture formulated by Lara
\cite{Jalr10, Jalr}, except the items concerning  the initial values.

Let us start with a definition. 
\begin{definition}\label{main_definitions} Let $a\in \Z_+$.
\begin{enumerate}
\item  We set
    \begin{align*}
    r_a = (q-1)p^m,
    \end{align*}
    where $m$ is the smallest integer such that $a\le p^m$.
\item Put
    \begin{align*}
    \phi(j):=r_a-a-j(q-1).
    \end{align*}
\item We define
    \begin{align*}
    j_{a, \max}=\left \lfloor \frac{r_a-a}{q-1} \right \rfloor.
    \end{align*}
\item Let $q$ be prime. For $0 \le j \le j_{a,\max}$, let $c_{a,j}\in \F_p $ be
defined by:
    \begin{align*}
        c_{a,j}=\begin{cases}
        1 & \mbox{ if } j =0\\
        \left \lceil \frac{j(q-1)} {j_{a, \max} } \right \rceil ^{-1}
\binom{r_a-a}{j(q-1)} & 0 < j \le j_{a, \max}.
    \end{cases}
\end{align*}
\item  For each $j,\;0\le j\le p-1$, let $\mu_j$ be the number of $j$'s in the
$p$ expansion of $a-1$. Set
\begin{align*}
t_a = \prod _{j=0}^{p-2} (p-j)^{\mu_j}.
\end{align*}
\end{enumerate}
\end{definition}

\begin{remark}
For $q$  prime and  $0 < j \le j_{a,\max}$, $\lceil j(q-1)/j_{a,\max}
\rceil$ is not zero in $\F_p$ because  $ 0 <  j(q-1)/j_{a,\max}  \le
q-1<p$. Therefore, $c_{a,j}$ in Definition~\ref{main_definitions} (4) makes
sense. For $q$ non-prime, $c_{a,j}$ is not always defined (eg., $q=4$, $a=5$,
$j = 3$).
\end{remark}

Let $a,b\in \Z_+$.  Let $j \in \set{0,1,\dotsc,p^m-a}$. Since $p^m$ and $q-1$
are coprime, $p^m$ is a unit in $\Z/(q-1)\Z$; so, the equation $j \equiv -ip^m
\bmod(q-1)$ has always a solution. There exists exactly one 
solution in the range $0 \le i <q-1$.

\begin{definition}\label{lj}
Let $i_j$ be the unique integer $0 \le i_j <q-1$ such that $j+i_j p^m
\equiv 0  \bmod (q-1)$. Let $l_j$ be the nonnegative integer defined by
\begin{align*}
  l_j = \frac{j+i_j p^m}{q-1}.
\end{align*}
\end{definition}

 In general, the correspondence $j\mapsto i_j$  between  $\{0,1,
\dotsc,p^m-a\}$ and  $\set{0,1,\dotsc,q-2}$ is neither
injective (eg., $q=3, a=4$) nor surjective (eg., $q=4, a=3$).

\begin{proposition}
  The map
\begin{align*}
\begin{array}{ccl}
\set{0,1,\dotsc,p^m-a} & \to &\set{l(q-1)\mid 0 \le l \le j_{a,\max} } \\
j & \mapsto & j+i_jp^m \\
\end{array}
\end{align*}
is injective.
\end{proposition}
\begin{proof}
Since $0 \le j \le p^m-a$ and $0 \le i_j \le q-2$, it follows that $0 \le j
+i_jp^m
\le p^m-a + p^m(q-2)= r_a-a$. This shows that the map is well defined.
  If $a=p^m$, then the map is clearly injective. Assume $a <p^m$.  If $j_1 + i_{j_1}p^m = j_2 + i_{j_2}p^m$,
then $p^m \mid j_1 - j_2$. Since $0 \le j_1,j_2 \le p^m-a<p^m$, we conclude that
$j_1 = j_2$.
\end{proof}

When $q = 2$, the map of the above proposition  is a bijection, but in general
the map is not surjective. For instance, consider $q = 3$ and $a = 2$.

\begin{proposition}\label{prop245}
The following statements hold in $\F_p$,
\begin{enumerate}
  \item [a)] If $0\le i \le q-1$, then
\begin{align}
  \binom{q-1}{i} = (-1)^i \label{binomial(q-1,i)}.
\end{align}
\item [b)] If $0 \le i \le p-2$, then
\begin{align*}
  \binom{p-2}{i} = (-1)^i(i+1).
\end{align*}

\item [c)] The number of $j$'s,  $0 \le j \le p^m-a$, such that
$\binom{p^m-a}{j}$ is not zero modulo $p$ is $t_a$.

\end{enumerate}
\end{proposition}
\begin{proof}
a) First consider the case $0 \le i \le p-1$. Then,
$(p-1)!$, $i!$, and $(p-1-i)!$ are nonzero modulo $p$; therefore, in $\F_p$, we
have
\begin{align*}
\binom{p-1}{i}& = \frac{(p-1)(p-2) \dotsm (p-i)}{i!} 
 =\frac {(-1)(-2) \dotsm (-i)}{i!} 
 = (-1)^i. 
\end{align*}
Thus, \eqref{binomial(q-1,i)} holds in this special case.

For general $i$, let $i = \sum _{l=0}^{s-1} i_l p^l$ be the base $p$ expansion
of $i$. The base $p$ expansion of  $q-1$ is $\sum_{l=0}^{s-1} (p-1)p^l$. Note
that $(-1)^{a_l} = (-1)^{a_lp^l}$.  By Lucas theorem, we conclude
\begin{align*}
\binom{q-1}{i} = \prod _{l=0}^{s-1} \binom{p-1}{i_l}  = \prod _{l=0}^{s-1}
(-1)^{i_lp^l} = (-1)^i.
\end{align*}
Thus, \eqref{binomial(q-1,i)} holds in the general case.

b) For the second part,  we have
\begin{align*}
  \binom{p-2}{i}  = \frac{(p-2) \dotsm (p-2-(i-1))}{i!} 
 = \frac{(-2) \dotsm (-(i+1))}{i!} 
 = (-1)^i(i+1).
\end{align*}

c)
Let $t_a'$ be the number of $j$'s in $\set{0,1,\dotsc,p^m-a}$   such
that $\binom{p^m-a}{j} \not\equiv 0 \bmod p$. 
Let us prove that $t_a'= t_a$. Since $a\le p^m$, then $a-1
\le p^m-1 = \sum _{l=0}^{m-1} (p-1)p^l$.   Let $a-1 = \sum_{l=0}^{m-1}a_lp^l$
and $j = \sum _{l=0}^{m-1}b_l p^l$ be the
base $p$ expansions of $a-1$ and $j$, respectively. Since $p^m-1-(a-1)=p^m-a$,
the base $p$ expansion of $p^m-a$ is $\sum _{l=0}^{m-1}(p-1-a_l)p^{l}$. By
Lucas theorem, 
\begin{align*}
  \binom{p^m-a } { j } \equiv
\prod _{l=0}^{m-1} \binom{p-1-a_l}{b_l} \bmod p.
\end{align*}
Since $\binom{\alpha}{\beta}$ vanishes modulo $p$ if $\beta>\alpha$, by choosing
$b_l$ in the  set $\{0,1,\dotsc,p-1-a_l\}$, we guarantee that
$\binom{p^m-a}{j}$ does not vanish modulo $p$. Therefore, 
\begin{align*}
  t_a'= \prod _{l=0}^{m-1}(p-a_l) = \prod _{k=0}^{p-2}(p-k)^{\mu_k} =
t_a.
\end{align*}
\end{proof}

\begin{proposition}
Let $q = 2$. Then,   $j_{a,\max}=0$ if and only if  $a=p^m$. If $q >2$, then
$j_{a,\max}=0$ if and only if $a=1$.
\end{proposition}

\begin{proof} This follows by a straight calculation from the definitions. 
\end{proof}

\begin{definition}\label{faj}
For each $j$, $0 \le j \le p^m-a$, let $f_{a,j}\in \F_p$  be defined by
\begin{align*}
f_{a,j} & = \binom{p^m-a}{j} (-1)^{j}. 
\end{align*}
\end{definition}

\begin{proposition}\label{faj_ne_0}
  The number of $j$'s such that $f_{a,j} \ne 0$ is $t_a$.
\end{proposition}
\begin{proof}
It follows immediately from Proposition~\ref{prop245} (c).
\end{proof}

When $q$ is prime, we have another description for $f_{a,j}$ which is
consistent with the Part 3 of the main conjecture in \cite{Jalr10} (See
the Remark~\ref{rem517} (1)).

\begin{proposition}\label{prop247}
If $q$ is prime (so that $q=p$), then for $j\in \set{1,2,\dotsc, p^m-a}$
\begin{align*}
  f_{a,j} = \left \lceil
\frac{j+i_j p^m}{j_{a,\max}} \right \rceil ^{-1} \binom{r_a-a}{j+i_j p^m} =
c_{a,l_j},
\end{align*}
where $c_{a,l_j}$ is as defined in \ref{main_definitions} and 
\ref{lj}. 
\end{proposition}
\begin{proof}
We claim that (A) $i_j+1 = \left \lceil \frac{j+i_jp^m}{j_{a,\max}} \right
\rceil$, which is equivalent to showing
$$i_jj_{a, \max} \leq i_jp^m < j+i_jp^m\leq (i_j+1)j_{a,\max}.$$
It is enough to  prove the rightmost inequality. To do so, we first compute
$i_j$ for $j>0$, as follows. Write $j = l(q-1)+r$, $0 \le r <q-1$. If $r>0$,
then
\begin{align*}
  j + (q-1-r)p^m = (l+p^m)(q-1) + r(1-p^m) \equiv 0
\bmod(q-1),
\end{align*}
and $0 < q-1-r<q-1$. If $r = 0$, then $j \equiv 0 \bmod (q-1)$. By definition,
it follows that $i_j=q-1-r$ if $r>0$ and $i_j = 0$, otherwise. Note 
$0 \le i_j \le q-2$,  so that $i_j+1$ is non-zero in $\F_p$.  Write
$a = l'(q-1)+a'$, $0\le a' <q-1$. Then $j_{a,\max} = p^m-l'-y$, where $y=0$ if
$a'=0$ and $y=1$ otherwise. Let us consider first the case $j\not \equiv 1
\bmod(q-1)$. Thus $i_j\ne q-2$.  Let $j_0 \in \set{2,\dotsc,p^m-a} \cap
\set{2,\dotsc,q-1}$ such that $i_j = i_{j_0}$. Then $j_0 =q-u$ for some $u$,
$1\le u \le q-2$ and $i_{j_0} = u-1$, $0 \le i_{j_0} <q-2$. Then 
\begin{align*}
j + i_{j} p^m & \le  j_{0} +i_{j_0}p^m + (q-1) \left \lfloor
\frac{p^m-a+i_jp^m -(j_0 +i_{j_0}p^m)}{q-1} \right \rfloor\\
& = q-u + (u-1)p^m +(q-1) \left \lfloor \frac{p^m-a-j_0}{q-1} \right \rfloor\\
& = q-u+(u-1)p^m+ p^m-1-l'(q-1)-z(q-1),
\end{align*}
where $z=1$ if $0< (a' +j_0-1)/(q-1) \le 1$ and $z=2$ if $1 <(a' +j_0-1)/(q-1)
<1 +(q-2)/(q-1)$. Since $u\le q-1$, then $l'u \le l'(q-1) \le l'(q-1) +
(z-1)(q-1)$. Therefore
\begin{align*}
  j + i_j p^m = up^m-u -l'(q-1)-(z-1)(q-1) \le up^m-l'u -yu =
(i_j+1)j_{a, \max}.
\end{align*}
On the other hand, if $j \equiv 1 \bmod(q-1)$, then $i_j+1=q-1$.
Since $j\le p^m-a$, then $j + i_j p^m \le r_a-a$. Now,
\begin{align*}
  \frac{j+i_jp^m}{i_j+1} = \frac{j+i_jp^m}{q-1} \le \frac{r_a-a}{q-1}.
\end{align*}
The claim now follows as the left side is an integer.

Let $j = b_0 + b_1 p + \dotsb + b_{m-1}p^{m-1}$ and $a-1 =
a_0 + a_1 p + \dotsb + a_{m-1}p^{m-1}$ be the base
$p$ expansions of $j$ and $a-1$, respectively. The base $p$ expansion of
$p^m(p-1)-1$ is $\sum _{l=0}^{m-1} (p-1)p^{l}
+ (p-2)p^m$. Consequently, the base $p$ expansion of $r_a-a = r_a-1 -(a-1)$ is
$
    \sum _{l=0}^{m-1}(p-1-a_l)p^l +(p-2)p^m.
$
Finally, the base $p$ expansion of $j+i_jp^m$ is $\sum_{l=0}^{m-1}b_l p^l +
i_jp^m$. 

Note $j+i_j \equiv j+ i_jp^m
\equiv 0 \bmod (q-1)$, 
so, (B)  $(-1)^{i_j}=(-1)^j$.  By the Lucas theorem,  b)  of
Proposition~\ref{prop245}, and (A), (B) 
respectively, we have 
\begin{align*}
  \binom{r_a-a}{j+i_jp^m} & = \binom{p-1-a_0}{b_0} \dotsm
\binom{p-1-a_{m-1}}{b_{m-1}} \binom{p-2}{i_j} \\
& =  \binom{p^m-a}{j} (-1)^{i_j} (i_j+1) \\
& = f_{a,j} \left \lceil
\frac{j+i_jp^m}{j_{a,\max}} \right \rceil.
\end{align*}
\end{proof}

\begin{proposition}\label{prop249}
Let  $q$ be arbitrary  and $a\in \Z_+$. For each  $n \in A_{1^+}$, let $g_n$ be
defined by
\begin{align*}
  g_n = - \frac{[1]^{p^m-a}}{n^{p^m-a}}[1]^aS_1(a).
\end{align*}
Then, 
\begin{align*}
  g_n = 1 + \sum _{j=1}^{p^m-a} f_{a,j} n^{r_a -(j+i_j p^m) }
\end{align*}
where $f_{a,j}$ is as in Definition~\ref{faj} and  $i_j$ is the unique integer,
$0 \le i_j <q-1$, such that $j+i_j p^m
\equiv 0  \bmod (q-1)$.
\end{proposition}
\begin{proof}
Note 
\begin{align*}
  -\frac{[1]^{p^m-a}}{n^{p^m-a}} 
= -(n^{q-1}-1)^{p^m-a} 
= \sum _{j_1=0}^{p^m-a} \binom{p^m-a}{j_1}(-1)^{p^m-a-j_1+1} n^{j_1(q-1)}.
\end{align*}
On the other hand, by specializing \cite[3.3]{Thakur_Multizeta08} to $d=1$ we
see that 
\begin{align*}
  S_1(k+1) =\frac{(-1)^{k+1}}{[1]^{k+1}} \left(
1 + \sum _{k_1=1}^{\lfloor k/q \rfloor} \binom{k-k_1(q-1)}{k_1}(-1)^{k_1}
[1]^{k_1(q-1)}
\right).
\end{align*}
Hence,
\begin{align*}
  [1]^a S_1(a) 
& = (-1)^a + \sum _{j_2=1} ^{n_a}
\binom{
a-1-j_2(q-1) }{j_2} (-1)^{a+j_2} n^{j_2(q-1)}(n^{q-1}-1)^{j_2(q-1)}.
\end{align*}
where $n_a = { \lfloor (a-1)/q \rfloor
}$. Since
\begin{multline*}
  n^{j_1(q-1)} n^{j_2(q-1)}(n^{q-1}-1)^{j_2(q-1)} \\
 = \sum_{j_3=0}^{j_2(q-1)} \binom{j_2(q-1)}{j_3} 
(-1)^{j_2(q-1)-j_3} n^{(j_1+j_2+ j_3)(q-1)},
\end{multline*}
we see that $g_n =  -(n^{q-1}-1)^{p^m-a} [1]^a
S_1(a)\in \F_p[n]$. Note that all powers of
$n$ involved are `even'.  If $n = t-\theta$ for some $\theta\in \F_q$, then
$g_n(t+\theta) =g_{t}$. If $g_n = \sum _{j} c_{n,j}n^{j(q-1)}$, it follows that
$c_{n,j} = c_{t,j}$ for all $j$. Now it suffices to compute $c_{t,j}$. To do so
let us put  $h_\theta = (t-\theta)^{p^m-a} \left(
(t-\theta)^{q-1}-1 \right)^{p^m}$, $\theta \in \F_q^*$.
 Then,
\begin{align*} 
h_\theta & = \sum _{j=0}^{p^m-a} \binom{p^m-a}{j} t^{p^m-a-j}\theta^j (-1)^j
\left(
\sum_{i=0}^{q-1} \binom{q-1}{i}^{p^m} t^{p^m(q-1-i)}\theta^{i p^m }(-1)^i
-1\right) \\
& = \sum _{j=0}^{p^m-a} f_{a,j} t^{p^m-a-j}
\left(
\sum_{i=0}^{q-1}  t^{p^m(q-1-i)}\theta^{j+ i p^m }
- \theta^j  \right).
\end{align*}
Hence, 
\begin{align*}
 \sum _{\theta \in \F_q^*} h_\theta & = 
\sum _{j=0}^{p^m-a} f_{a,j} t^{p^m-a-j}
\left(
\sum_{i=0}^{q-1}  t^{p^m(q-1-i)}  \sum
_{\theta\in \F_q^*} \theta^{j+ip^m} -  \sum _{\theta\in \F_q^*} \theta^j
\right) \\
& = \sum _{j=0}^{p^m-a} f_{a,j} t^{p^m-a-j}
\left(
\sum_{i=0}^{q-2}  t^{p^m(q-1-i)}  \sum
_{\theta\in \F_q^*} \theta^{j+ip^m} 
\right)\\
& =  \sum _{j=0}^{p^m-a} f_{a,j} t^{p^m-a-j}
\left(
  t^{p^m(q-1-i_j)} (-1)  
\right)\\
& = - \sum _{j=0}^{p^m-a} f_{a,j} t^{r_a-a -(j+i_jp^m)+p^m}
\end{align*}
Here, we used that $\sum_{\theta\in \F_q^*} \theta^l$ is 0 if $q-1$ does not
divide $l$, and $-1$ if $l\ge 0$ is divisible by $q-1$. Now, 
\begin{align*}
[1]^{p^m}S_1(a) 
 & = \sum _{n\in A_{1^+}} \frac{[1]^{p^m}}{n^a}
 =  \sum _{n\in A_{1^+}} n^{p^m-a} \left(n^{q-1} -1\right)^{p^m} \\
 &= t^{p^m-a}(t^{r_a}-1) + \sum _{\theta \in \F_q^*} h_\theta.
\end{align*}
It follows that
\begin{align*}
g_t & 
 = -t^{r_a}+1 + \sum _{j=0}^{p^m-a} f_{a,j} t^{r_a-(j+i_jp^m)} 
 =  1 + \sum _{j=1}^{p^m-a} f_{a,j}
t^{r_a-(j+i_jp^m)}.
\end{align*}
\end{proof}

\begin{example}
  Let $q=9$ and $a = 17$. Then, $m = 3$ and $r_a = 27(8) = 216$. Since $a-1 = 1
+2p + p^2$, $t_a =3^0 2^2=4$. The $j$'s for which $f_{a,j}\ne 0$ are 0, 1, 9,
and 10. Using that $0\mapsto 0$, $1\mapsto 5$, $9\mapsto 5$, and $10\mapsto
2$, we have
\begin{align*}
  g_t = 1 + f_{a,1}t^{216-136} + f_{a,9}t^{216-144} + f_{a,10}t^{216-64} 
 = 1 +2t^{80} + 2t^{72} + t^{152}.
\end{align*}
\end{example}

\begin{definition}
Let $q$ be arbitrary. For  $a,b\in \Z_+$, let $S(a,b)$ denote the  pairs
$(f_i,a_i)$ such \eqref{Delta1} holds for $d = 1$. 
\end{definition}

The following theorem proves parts 1 and 7 of the main conjecture 1.5 in
\cite{Jalr10}, and is more precise in that it gives a complete recipe 
for general $q$. 

\begin{theorem}\label{thm2412}
Let $a,b\in \Z_+$. Let $r_a, i_j, f_{a, j}$  be as in definitions
\ref{main_definitions}, \ref{lj}, \ref{faj}. Then
\begin{align*}
  \Delta(a,b+r_a) - \Delta(a,b) = \sum _{j=0}^{p^m-a} f_{a,j}
S_1(a+b+(j+i_jp^m)).
\end{align*}
In particular, the sets
$S(a,b)$ can be found recursively with recursion length $r_a$, by
\begin{align*}
S(a,b+r_a )= S(a,b)\cup T(a,b+r_a),
\end{align*}
where
\begin{align*}
  T(a,b+r_a) = \set{ \left( f_{a,j}, a+b+(j+i_jp^m)  \right) \mid 0 \le j \le
p^m-a, f_{a,j}\ne 0 },
\end{align*}
or equivalently, 
\begin{align*}
T(a,b+r_a)  = \set{\left( f_{a,j}, b+r_a - \phi(l_j)  \right) \mid 0 \le j \le
p^m-a, f_{a,j}\ne 0   },
\end{align*}
where  $\phi(l_j) = r_a-a-l_j(q-1)$.

 The set $T(a,b+r_a)$ is a set of size $t_a$.
\end{theorem}
\begin{proof}

By definition of $\Delta(a,b+r_a)$, it follows that
\begin{align*}
\Delta(a,b+r_a)  = \sum _{ \substack{ {n_1\ne n_2}\\ {n_1,n_2\in A_{1^+}}} }
\frac{1}{n_1^a
n_2^{b+r_a}}
 =\sum _{n_2\in A_{1^+}} \frac{1}{n_2^{b+r_a}} \left( S_1(a) -\frac{1}{n_2^a}
\right).
\end{align*}
By Proposition~\ref{prop249}, for any
$n\in A_{1^+}$, we have 
$
  g_n = 1 + \sum _{j=1}^{p^m-a} f_{a,j}n^{r_a-(j+i_jp^m)}. 
$
Let $\Sigma = \sum_{j=1}^{p^m-a} {f_{a,j}}S_1(b+r_a-\phi(l_j))$. Then,
\begin{align*}
 \Sigma   =\sum _{n_2\in A_{1^+}} \sum_{j=1}^{p^m-a}
\frac{f_{a,j}}{n_2^{b+r_a-(r_a-a +(j+i_jp^m))}} 
 &= \sum _{n_2\in A_{1^+}} \frac{1}{n_2^{b+r_a}} \frac{1}{n_2^{a}} \sum
_{j=1}^{p^m-a} f_{a,j} n_2^{r_a-(j+i_jp^m)}.
\end{align*}
Using this leads to
\begin{align*}
  \Delta(a,b+r_a) - \Sigma 
& = \sum _{n_2 \in A_{1^+}} \frac{S_1(a)}{n_2^{b+r_a}} 
\left( 1 -\frac{1}{S_1(a)n_2^a} \left( 1 +  \sum_{j=1}^{p^m-a} f_{a,j}
n_2^{r_a-(j+i_jp^m)}
\right) \right)\\
& =  \sum _{n_2 \in A_{1^+}} \frac{S_1(a)}{n_2^{b+r_a}} 
\left( 1 -\frac{1}{S_1(a)n_2^a} g_{n_2} \right)\\
& = \sum _{n_2 \in A_{1^+}} \frac{S_1(a)}{n_2^{b+r_a}} \left( 1
+\frac{[1]}{n_2} \right)^{p^m}\\
& = \sum _{n_2 \in A_{1^+}} \frac{S_1(a)}{n_2^{b+r_a}} \left( n_2^{q-1}
\right)^{p^m}\\ 
& = S_1(a)S_1(b).
\end{align*}

Therefore, $\Delta(a,b+r_a) -S_1(a)S_1(b) = \Sigma$. From this, we get
\begin{align*}
  \Delta(a,b+r_a) - \Delta(a,b) 
& = \Sigma + S_1(a+b)\\
& = \sum_{j=1}^{p^m-a} {f_{a,j}}S_1(b+r_a-\phi\left(l_j \right)
)) + S_1(b+r_a - \phi(0))\\
& = \sum_{j=0}^{p^m-a} {f_{a,j}}S_1(b+r_a-\phi\left(l_j \right)
).
\end{align*}
This shows that $T(a,b+r_a)$ is exactly as claimed.  By
Proposition~\ref{faj_ne_0} follows that  the size of $T(a,b+r_a)$ is precisely
$t_a$.
\end{proof}

\begin{remark}
\begin{enumerate}
  \item The set $T_a$ of pairs $\left( f_{a,j},  \phi(l_j)  \right)$
with $f_{a,j}\ne 0$  clearly is independent of $b$ as predicted in Conjecture
1.5 (1c) in \cite{Jalr10}, which did not have precise description for it as 
above. 

  \item The number of terms $t_a$ to be added at each step of the recursion 
depends on $q$ only through  $p$.
\end{enumerate}
\end{remark}

\begin{examples}[Special large indices, page 2338 \cite{Thakur_Multizeta08}]
Let $q = 2$.
\begin{enumerate}
  \item Let $a = 2^n-1$. Then, $m = n$ and $a-1 = 2 +\dotsb +
2^{n-1}$. There is only one zero in the base 2 expansion of $a-1$. Thus, $t_{a}
=
(2-0)^1=2$. At each step of recursion, two terms must be added.
Furthermore,
 $T_a = \set{(1,r_a-a), (1,r_a-a-1) }$ because in this case $l_j =j$ for
$j = 0,1$.
\item For $a = 2^n+1$,  $2^n$ top terms are added. Then, $m = n+1$. The base 2
expansion of $a-1$ has $n$ zeros and so $t_a = 2^n$. As before, for any $j$, $0
 \le j \le 2^n-1$, $i_j = 0$, and $l_j = j$. Thus, $T_a = \set{(1, r_a-a-j) \mid
0 \le j \le 2^n-1}$. 

\end{enumerate}
\end{examples}

\begin{example}[Example page 2337, \cite{Thakur_Multizeta08}]\label{ex2414}
  Let $q = 2$ and $a = 19$. Then, $m = 5$, $t_{19} = 8$, and $r_{19}=32$. In
this
case, $i_j = 0$, and $l_j = j$ for all $j$. The $j$'s for which $f_{19,j}\ne 0$
are 0, 1, 4, 5, 8, 9, 12, and 13. The polynomial $g_t$ is $g_t =1 +   t^{31} +
t^{28}
+ t^{27} + t^{24} + t^{23} + t^{20} + t^{19}$. 
Therefore,
\begin{align*}
  \Delta(19,b+32) - \Delta(19,b) & =  \sum_{j=0}^{13} f_{19,j} S_1(19+b+j) \\
 & =  S_1(b+19) + S_1(b+20) + S_1(b+23) + S_1(b+24) \\
 & \phantom{=}  +  S_1(b+27) + S_1(b+28) + S_1(b+31) + S_1(b+32).
\end{align*}
\end{example}

The following  corollary  proves parts 2, 5, and 6 of
the main conjecture in \cite{Jalr10}.

\begin{corollary}
Let notations be the same as before.
\begin{enumerate}
  \item [a)] $(1,\phi(0)) = (1,r_a-a) \in T_a$. 
\item [b)]
$T_a = \set{(1,\phi(0))}$ if and
only if $a =p^m$.
\item [c)] If $a'= p^{m'}a$, then
\begin{align*}
  T_{a'} = \set{ \left(f_{a,j}, p^{m'}\phi\left( l_j  \right) \right)
\mid 0 \le j \le p^m-a, f_{a,j}\ne 0 }.
\end{align*}

\item [d)] If $q$ is prime,
\begin{align*}
  T_{a} = \set{ \left(c_{a,l_j}, \phi\left( l_j  \right) \right)
\mid 0 \le j \le p^m-a, f_{a,j}\ne 0 }.
\end{align*}

\end{enumerate}

\end{corollary}
\begin{proof}
To prove a), just note that $f_{a,0}=1$.

b) When $a = p^m$, then $g_n = 1$ and, thus, $T_a = \{(1,r_a-a) \}$.
Conversely, since $t_a = 1$, by definition
of $t_a$, it follows that $a-1 = \sum_{i=0}^{m-1}(p-1)p^i= p^m-1$.

c) Let $a' = p^{m'}a$ for some integer $m'\in \Z_+$. Then, $m+m'$ is the
smallest integer such that $a'\le p^{m+m'}$. We have

\begin{align*}
  - \frac{[1]^{p^{m+m'}}}{n^{p^{m+m'}-a'}} S_1(a')
& = \left( - \frac{[1]^{p^{m}}}{n^{p^{m}-a}} S_1(a) \right)^{p^{m'}} \\
& = \left( g_n \right)^{p^{m'}}\\
& = \left( 1 + \sum _{j=1}^{p^m-a} f_{a,j}n^{r_a- (j+i_jp^m)} \right)^{p^{m'}}
\\
& = 1 + \sum _{j=1} f_{a,j}^{p^{m'}}  n^{ {p^{m'}(r_a- (j+i_jp^m)) } } \\
& = 1 + \sum _{j=1} f_{a,j}  n^{ {p^{m'}(r_a- (j+i_jp^m)) } }.
\end{align*}

d)  If $q$ is prime, by Proposition~\ref{prop247}  $f_{a,j} =
c_{a,l_j}$ follows. 
\end{proof}

\begin{proposition}\label{prop516}
  If $\binom{r_a-a}{l(q-1)} \ne 0$ in $\F_p$ for some $l$, $0 \le l \le
j_{a,\max}$, then there exists $j$, $0 \le j \le p^m-a$, such that $l(q-1) = j +
i_j p^m$ and $\binom{p^m-a}{j} \ne 0$.
\end{proposition}
\begin{proof}
  Writing the base $p$ expansion of $a-1$ as $\sum a_kp^k$ as before, we have 
$$r_a -a 
= p^m-a +p^m(q-2)
= \sum _{k=0}^{m-1}(p-1-a_k)p^k  + (p-2)p^m  + \sum _{k=m+1}^{m+s-1}(p-1)p^k. $$
Let $l(q-1) = \sum _{k=0}^{m+s-1} b_k p^k$ be the base $p$ expansion of
$l(q-1)$.
Since $\binom{r_a-a}{l(q-1)} \ne 0$, by Lucas theorem, we have
$\binom{p-1-a_k}{b_k} \ne 0$ for  $k=0,1,\dotsc, m-1$  and $\binom{p-2}{b_m} \ne
0$.  Therefore, $b_k \le p-1-a_k$ for $k = 0,1,\dotsc, m-1$ and $b_m \le p-2$.
Let $j = \sum _{k=0}^{m-1} b_kp^k$ and $i = \sum _{k=0}^{s-1}b_{m+k} p^{k}$.
Thus, $0 \le j \le p^m-a$ and $0 \le i \le q-2$. Since $j + ip^m = l(q-1) \equiv
0 \bmod (q-1)$, we have $i = i_j$. It follows that $\binom{p^m-a}{j} \ne 0$.
\end{proof}

\begin{remark}\label{rem517}
  \begin{enumerate}
    \item This proposition proves  Part 3 of the
main conjecture in \cite{Jalr10}: If there is no carry over
base $p$ in the sum of $l(q-1)$ and $\phi(l)$, then $\binom{r_a-a}{l(q-1)} \ne
0$ and, so, by  Proposition~\ref{prop516}  and Theorem~\ref{thm2412}, $(f_{a,j},
\phi(l_j))$ belongs to $T_a$.

\item By Proposition~\ref{prop247}, if $q$ is prime, $f_{a,j}\ne 0$ if and only
if $\binom{r_a-a}{l_j(q-1)}\ne 0$. If $q$ is not prime, we could have 
$\binom{p^m-a}{j} \ne 0$ and $\binom{r_a-a}{l_j(q-1)} = 0$ (eg., $q=4$, $a=5$,
$j=1$).
  \end{enumerate}
\end{remark}

\begin{example}[Theorems 3 and 7,\cite{Thakur_Multizeta08}]
We give another proof of Theorems 3 and 7 in \cite{Thakur_Multizeta08} using
Theorem~\ref{thm2412}. Let $q = 2$.
\begin{enumerate}
  \item [a)] Let $a = 1$. Then $m = 0$, $r_1 = 1$, $g_t = 1$, and
$t_1 = 1$. By Theorem~\ref{thm2412}, we have $\Delta(1,b+1)- \Delta(1,b) =
S_1(1+b)$. Then, $\Delta(1,2) = \Delta(1,1) + S_1(2)=S_1(2)$. By repeating this
process, we get $\Delta(1,3) = \Delta(1,2) + S_1(3) = S_1(2) + S_1(3)$. It
follows
that $\Delta(1,b) = \sum _{i=2}^b S_1(i)$. Therefore, we get
\begin{align*}
  \zeta(1)\zeta(b) =\zeta(1+b) + \sum _{i=1}^{b-1} \zeta(i,b+1-i).
\end{align*}

\item [b)] Now let  $a = 2$. Then $r_2 = 2$, $g_t = 1$, and $t_2=1$. Since
$\Delta(2,1) = S_1(2)$ and $\Delta(2,2)=0$, proceeding as in part a), it
follows that 
\begin{align*}
  \Delta(2,b) = 
\begin{cases}
 \sum _{i=1}^{(b-1)/2} S_1(2i+1) + S_1(2) & \mbox{if $b$ is odd,} \\
 \sum _{i=2}^{b/2} S_1(2i)  &\mbox{if $b$ is even.}
\end{cases}
\end{align*}

\end{enumerate}
\end{example}

\section{Closed formulas }

Let $a,b\in \Z_+$. By Theorem~\ref{shuffle_thm1}, there are $f_i \in \F_p$ and
$a_i\in \Z_+$ so that $\Delta(a,b) =\sum _{i=0}^{\delta} f_i S_1(a_i)$. From
this, we get the partial fraction decomposition of $\Delta(a,b)$:
\begin{align}
  \Delta(a,b) = \sum _{\mu \in  \F_q} \frac{h_\mu(t)}{(t-\mu)^{n}},
\label{pfd1}
\end{align}
where $h_\mu(t) = f_0 + f_1(t-\mu)^{a_0-a_1} +
\dotsb + f_{a_\delta}(t-\mu)^{a_0-a_\delta}$ (we assume that $n= a_0 > a_1
>\dotsb > a_\delta$). Uniqueness of the partial
fraction decomposition   guarantees uniqueness of $f_i$'s.

Conversely, since the denominator of  $\Delta(a,b)$ is a power of $[1]$, its
partial fraction decomposition is of the form \eqref{pfd1}, where $h_\mu(t)\in
\F_q[t]$ is  relatively prime to 
$t-\mu$; also, $\deg h_\mu <n$.
Now, $\Delta(a,b)$ is invariant with respect to the automorphisms $t
\rightarrow t+\theta$, $\theta \in \F_q$ of $A$. By the uniqueness of the
partial fraction decomposition, we have $h_{0}(t) =h_{\theta}(t+\theta)$ for
any $\theta \in \F_q$. Let $h_0(t) = f_{0} + f_{1} t + \dotsb +
f_{n-1}t^{n-1}$. Then, 
\begin{align*}
  \Delta(a,b) 
 = \sum _{\mu \in \F_q} \frac{h_0(t-\mu)}{(t-\mu)^n} 
 = \sum _{i=0}^{n-1} f_i\sum _{\mu \in \F_q} \frac{1}{(t-\mu)^{i}}
 = \sum _{i=0}^{n-1} f_i S_1(n-i).
\end{align*}

By the uniqueness of $f_i$'s and by Theorem~\ref{shuffle_thm1}, it follows that
$f_i \in \F_p$.

We proceed as follows to find  $h_0(t)$. Since ${[1]^n}/{t^n}$ is a unit modulo 
$t^n$, it follows from  \eqref{pfd1} that modulo $t^n$, we have 
$[1]^n \Delta(a,b)= \frac{[1]^n}{t^n} h_0(t)$, so that
 $ h_0(t) \bmod t^n  = \left( [1]^n \Delta(a,b) \bmod t^n
\right)\left(\frac{[1]^n}{t^n}  \bmod t^n \right)^{-1}$. To finish, we pick the
unique representative of $h_0(t) \bmod t^n$ of degree less than $n$.

We first explain the case $q = 2$, when we can simplify a lot to get a nice 
expression, before giving the general case.  When $a = b$, we know that $\Delta(a,b)=0$ 
\cite[Theorem 8]{Thakur_PowerSums08}. Since  $S(a,b) = S(b,a)$,
 we can assume, without loss of generality, that $a > b$.

\begin{theorem}\label{initial_values_q2}
  Let $q = 2$. If $a >b \ge 1$, then
\begin{align*}
  \Delta(a,b)  = \sum _{k=0}^{a-1} f_{k} S_1(a-k) \mbox{ where }
f_{k} = \sum _{\substack{ i+j = k \\ i \le 2^m-a \\ j \le a-b-1 }}
\binom{2^m-a}{i} \binom{a-b}{j}.
\end{align*}
In particular, 
\begin{align*}
  S(a,b) & = \set{ (f_{k},a-k) \mid f_k \ne 0,\, 0\le k \le a-1 }.
\end{align*}
\end{theorem}
\begin{proof}

Since $a>b$,
\begin{align*}
  \Delta(a,b) = \frac{1}{t^a (t+1)^b} + \frac{1}{(t+1)^a t^b} 
 = \frac{(t+1)^{a-b} + t^{a-b}}{t^a(t+1)^a} 
 = \frac{S_1(b-a)}{[1]^a}.
\end{align*} 
The degree of the numerator of $\Delta(a,b)$ is less than $a-b$ and, thus, less
than the degree of the denominator. Next, we apply the method explained above to
write   $\Delta(a,b)$ as a sum $\frac{h_0(t)}{t^a} +
\frac{h_0(t-1)}{(t-1)^a}$. Since $(t+1)^{2^m-a} (t+1)^a \equiv 1 \bmod
t^a$ we have that $(t+1)^{2^m-a} \bmod t^a$ is the inverse of $(t+1)^a \bmod
t^a$.  Then, $h_0(t) \bmod t^a =  (t+1)^{2^m-a} S_1(b-a) \bmod t^a$. The only
representative of degree less than $a$ of $(t+1)^{2^m-a} \bmod
t^a$ is  $(t+1)^{2^m-a}$ because $a>2^{m-1}$. The only representative of
degree less than $a$ of $S_1(b-a)\bmod t^a$ is $S_1(b-a)$. Now,
\begin{align*}
  (t+1)^{2^m-a} & = \sum _{i=0}^{2^m-a} \binom{2^m-a}{i}t^i,\\
S_1(b-a) & = t^{a-b} + (t+1)^{a-b} = \sum _{j = 0}^{a-b-1} \binom{a-b}{j}  t^j.
\end{align*}

The degree of $(t+1)^{2^m-a} S_1(b-a)$ is at
most $2^m-b-1$. By  taking as $h_0(t)$ the remainder of $(t+1)^{2^m-a}S_1(b-a)$
when divided by $t^a$ the theorem follows.
\end{proof}

\begin{remark}
\begin{enumerate}

\item If $b \ge 2^m-a$, then $2^m-b-1<a$ and, therefore, it is  not necessary to
divide by $t^a$.

  \item If $k = 0$, then $ i = j = 0$ and $f_0 = \binom{2^m-a}{0}
\binom{a-b}{0}=1$.
\end{enumerate}
\end{remark}

Now we make the recipe in Theorem~\ref{shuffle_thm1}
explicit by giving a closed formula for $f_i$ there.

\begin{theorem}\label{closed_formula_for_general_q}
  Let $q$ be arbitrary. 
Let $a,b\in \Z_+$ such that $a \ge b \ge 1$; let $m$ the
smallest integer such that $a  \le p^m$. 
Then
\begin{align}
\Delta(a,b) = \sum _{i=0}^{a-1} f_i S_1(a-i),\label{delta_as_linear_combination}
\end{align}
where $f(t):= f_0 + \dotsb + f_{a-1}t^{a-1}\in \F_p[t]$ is given by
\begin{align*}
- \left( t^{q-1}-1 \right)^{p^m-a}
\sum _{\theta \in \F_q}
\sum _{\mu \in \F_q^*}
\left(
\sum_{j=1}^{q-1}  \mu^{q-1-j} (t+\theta)^{j-1}
\right)^a
(t+\theta-\mu)^{a-b} \bmod t^a.
\end{align*}

Equivalently, $f(t)$ is given by
\begin{multline}
\sum _{i_3=0}^{p^m-a}
\sum _{k}
\sum _{i_1=0}^{a-b}
\sum _{i_2=0}^{\sigma(k)+i_1-1}
\binom{p^m-a}{i_3}
\binom{a}{k_1,\dotsc,k_{q-1}}
\binom{a-b}{i_1} \times \\
\binom{\sigma(k) +i_1}{i_2} 
(-1)^{b+i_1+i_3}
t^{i_2+i_3(q-1)},\label{gf_extended}
\end{multline}
where the second sum extends over all $(q-1)$-tuples $k = (k_1,\dotsc,k_{q-1})$
of non-negative integers such that $k_1 +\dotsb +k_{q-1}=a$; $\sigma(k):=\sum_{j=2}^{q-1}
(j-1)k_j$; $\tau(k):= \sum_{j=1}^{q-2} (q-1-j)k_j$; and $i_1, i_2, i_3$ are
subjected to  
$\sigma(k)+i_1-i_2$,  $\tau(k)+a-b-i_1$ being both `even', and
$i_2+i_3(q-1)<a$.
\end{theorem}
\begin{proof}
The proof method is the same as the one of Theorem~\ref{initial_values_q2}, but
with more combinatorial complications as we deal with any $q$. We now sketch the steps, 
omitting the routine calculations. First,  the inverse modulo $t^a$ of
$[1]^a/t^a$ is 
$-(t^{q-1}-1)^{p^m-a}$ responsible for the first binomial coefficient. Let $\theta, \mu
\in\F_q$, with $\mu\neq 0$. Then raising
\begin{align*}
\frac{[1]}{(t+\theta)(t+\theta-\mu)}
& = 
\frac{\left(t+\theta-\mu\right)^{q-1}-1}{t+\theta}
= 
\sum _{j=1}^{q-1} \binom{q-1}{j} (t+\theta)^{j-1} (-\mu)^{q-1-j}\\
& = \sum_{j=1}^{q-1}\mu^{q-1-j}(t+\theta)^{j-1},
\end{align*}
to the $a$-th power by the multinomial theorem
brings 
in $\tau(k), \sigma(k)$, the multinomial coefficient; whereas the  multiplication by 
$(t+\theta-\mu)^{a-b}$ the next binomial coefficient, and also $(t+\theta)^{\sigma(k)+i_1}$ 
bringing in the last binomial coefficient. Finally,  we sum over $\theta$ and $\mu$ and use the fact that 
$\sum_{\xi\in\F_q^*}\xi^\ell$ is $-1$ or $0$ according as $\ell$ is `even' or
not, accounting for the conditions. 
\end{proof}

\begin{corollary}\label{parity_restriction}
The $i$'s in equation~\eqref{delta_as_linear_combination} are such
that $b+i$ is `even'.
\end{corollary}
\begin{proof}
  Each $i$ is of the form $i_2+i_3(q-1)$. Since both $\sum
_{j=1}^{q-1} (j-1)k_j+i_1-i_2$ and $\sum _{j=1}^{q-1} (q-1-j)k_j +a-b-i_1$ are
`even', and $\sum _{j=1}^{q-1} (j-1)k_j + \sum _{j=1}^{q-1} (q-1-j)k_j + a =
a(q-1)$, it follows that $ a(q-1)-(b+i_2) $ is `even'. Then,
$b+i_2$ is `even', too.  Therefore, $b+ i=b+i_2 + i_3(q-1)$ is
`even'.
\end{proof}

\begin{remark}
\begin{enumerate}
  \item  Corollary~\ref{parity_restriction} proves the
parity conjecture of Thakur \cite[5.3]{Thakur_Multizeta08}. We have already
proved the parity conjecture by a different method \cite{Jalr11b}.

\item  When $q=2$, in equation~\eqref{gf_extended} instead of fours sums
and four
multinomial coefficients, we have 
two sums and two binomial coefficients  and we reduce to formula for $q=2$.

\item By Lucas theorem the multinomial coefficient or the binomial coefficients
in \eqref{gf_extended} are zero if there is carry over base $p$ in the
corresponding sum. So only terms where there is no carry over base $p$ (in the
corresponding sums) need to be considered. Similarly, the other conditions
and vanishings of binomial coefficients reduce the number of terms in the sum a lot. 
\end{enumerate}
\end{remark}

\begin{example}[Both indices large, page 2338, \cite{Thakur_Multizeta08}]
  We  use Theorem~\ref{initial_values_q2} to prove two conjectures
due to Thakur. Let $q = 2$.
\begin{enumerate}
  \item [a)] Let $a = 2^{n}+1$ and $b = 2^{n}-1$. Then, 
\begin{align*}
  \Delta_d(a,b) & = \sum _{k=2}^{2^{n}+1} S_d(k,2^{n}+1-k).
  \end{align*}

\item [b)] Let $a = 2^n+1$ and $b = 2^{n-1}$. Then,
\begin{align*}
  \Delta_d(a,b) = S_d(a,b) + \sum _{i=k}^{2^{n-1}+1} S_d(k,
3\cdot 2^{n-1}+1-k).
\end{align*}
\end{enumerate}
\begin{proof}
It is enough to prove for the case $d = 1$ as it is established in
Theorem~\ref{shuffle_thm1}.

  a) In this case, $m = n+1$, $p^m-a = 2^n-1$, and $a-b-1 =
1$. Then, $f_k = 1$ for $k = 1,\dotsc, 2^n-1$ and $f_{a-1}=0$ for $k = a-1=2^n$.
By Theorem~\ref{initial_values_q2},
\begin{align*}
  \Delta(2^{n}+1,2^{n}-1) & = \sum _{k=0}^{2^{n}-1} S_1(2^{n}+1-k).
\end{align*}

b) Now, $m = n+1$, $p^m-a = 2^n-1$,
and $a-b = 2^{n-1}+1$. Looking at the base 2 expansions of $j$ and $2^{n-1}+1$
and using  Lucas theorem, we see that the values of $j$, $0\le j
\le 2^{n-1}$, such that $\binom{2^{n-1}+1}{j} \ne 0$ are $j = 0,1,2^{n-1}$. On
the other hand, $\binom{2^n-1}{i} \ne 0$ for $i = 0,1, \dotsc, 2^n-1$. We
already know that $f_0 = 1$. For $1\le k \le 2^{n}-1$, we have
\begin{align*}
  f_k = \binom{2^n-1}{k}\binom{2^{n-1}+1}{0}
+\binom{2^n-1}{k-1}\binom{2^{n-1}+1}{1} + \binom{2^n-1}{k-2^{n-1}}
\binom{2^{n-1}+1}{2^{n-1}}.
\end{align*}
Therefore,  $f_k =0$ for $1 \le k <2^{n-1}$ and $f_k = 1$ for $2^{n-1} \le k \le
2^{n}-1$. For $k = a-1 = 2^n$,
\begin{align*}
  f_k = \binom{2^n-1}{k-1}\binom{2^{n-1}+1}{1} + \binom{2^n-1}{k-2^{n-1}}
\binom{2^{n-1}+1}{2^{n-1}} = 0.
\end{align*}
It follows that $S(a,b)  = \set{(f_0,a-0)} \cup \set{ (f_{k},a-k) \mid
2^{n-1} \le k \le 2^n-1 }$. Then
\begin{align*}
  \Delta(a,b)& = S_1(2^n+1) + \sum _{k = 2^{n-1}}^{2^{n}-1}S_1(2^n+1-k).
\end{align*}
\end{proof}
\end{example}

The following corollary gives special cases of the
Theorem~\ref{initial_values_q2}.

\begin{corollary}\label{cor2421}
Let $q=2$.
\begin{enumerate}
  \item [a)] If $a = 2^m$ and $b<2^m$, then 
$$\Delta(2^m,b) = \sum _{k = 0}^{2^m-b-1}
\binom{2^m-b}{k} S_1(a-k).$$
$\binom{2^m-b}{k} \ne 0$ if there is no carry over base 2 in the sum of $k$ and
$b-1$.
\item [b)] If $b = a-1$, then
\begin{align*}
  \Delta(a,a-1)  = \sum _{k = 0}^{2^m-a} \binom{2^m-a}{k} S_1(a-k).
\end{align*}
The number of nonzero coefficients is $t_a$.

\end{enumerate}

\end{corollary}
\begin{proof}
a) If $a = 2^m$, then $i = 0$ and $\binom{2^m-a}{0} = 1$. Thus, $f_{k} =
\binom{2^m-b}{k}$ for $0 \le k \le 2^m-b-1$. Let $b-1 = b_0 + b_1 2 +  \dotsb +
b_{m-1}2^{m-1}$ and $k = k_0 + k_12 + \dotsb + k_{m-1}2^{m-1}$   be the base $p$
expansions of $b-1$ and $k$, respectively. Then the base 2 expansion
of $2^m-b$ is $\sum (1-b_l)p^l$. If $k \le 2^m-b-1$ and there is no
carry over base 2 in the sum of $k$ and $b-1$, then $f_{k} \ne 0$ because of
Lucas theorem.

b) If $b = a-1$, then $a-b-1 =0$ and, so, $j = 0$. Therefore,
$f_{k}=\binom{2^m-a}{k}$ for $0 \le k  \le 2^m-a$. In this special case, the
number of $f_{k} \ne 0$ is $t_a$ because of Proposition~\ref{prop245} (c).
\end{proof}

\begin{example}
We use   the corollary to prove  the following conjectures for $q=2$
\cite[Section 4.1.3]{Thakur_Multizeta08}:
\begin{align}
  \Delta_d(2^n,2^n-1) & =  S_d(2^n,2^n-1) \label{sec413a} \\
 \Delta_d(2^{n}+1,2^{n}) & =  \sum _{i=2}^{2^{n}+1} S_d(i,2^{n+1}+1-i).
\label{sec413b}
\end{align}
When $l<n-1$,
\begin{multline}
  \Delta_d(2^n-2^l, 2^n-2^l-1) = S_d(2^n-2^l,2^n-2^l-1) \\ + 
S_d(2^n-2^{l+1},2^n-1). \label{sec413d}
\end{multline}

In all cases, it is enough to prove for $d = 1$. Conjecture~\ref{sec413a} 
follows immediately either from  a) or b) of Corollary~\ref{cor2421}, taking $a
=
2^n$ and $b = a-1$. For Conjecture~\ref{sec413b},  take $a = 2^{n}+1$ and $b =
2^{n}$. Then  $m = n+1$. By Proposition~\ref{prop247} (b), we have
$\binom{2^{n}-1}{k} =1$ for $k=0,1,\dotsc,2^{n}-1$;  so,
\begin{align*}
  \Delta(2^{n}+1, 2^{n})& = \sum _{k=0}^{2^{n}-1} S_1(2^{n}+1-k)
 = \sum _{k=2}^{2^{n}+1} S_1(k).
\end{align*}
Finally, for Conjecture~\ref{sec413d}, let $a = 2^n-2^l$ and $b = a-1$. Then, $m
= n$, $p^m-a = 2^l$ and, thus,
\begin{align*}
  \Delta(a,b) & = \sum _{k =0}^{2^l} \binom{2^l}{k} S_1(a-k) \\
& = S_1(2^n-2^l) + S_1(2^n-2^l-2^l) + \sum _{k =1}^{2^l-1} \binom{2^l}{k}
S_1(a-k).
\end{align*}
The 2-adic valuation of $\binom{2^l}{k}$ is $\ell(k) + \ell(2^l-k) -\ell(2^l)$,
where $\ell(k)$ is the sum of the digits of $k$ base $q=2$. Since $k\ne 0$ and
$2^l-k\ne 0$, we have that $\ell(k) + \ell(2^l-k) -\ell(2^l) \ge 1 +1 -1=1$ and
 $\binom{2^l}{k}=0$ for $k = 1,\dotsc, 2^l-1$ and the result follows.

\end{example}

\begin{example}[Conjecture 2.8.1, \cite{Jalr10}] Now we prove that
\begin{align}
  \Delta_d(q^n,q^n-1) = -S_d(q^n),\label{Conjecture281}
\end{align}
for any $q$, which is a generalization of \eqref{sec413a}. 
In \cite{Jalr11b}, we have already proved \eqref{Conjecture281}. Here, we prove
it again, as a corollary of
Theorem~\ref{closed_formula_for_general_q}. By Theorem~\ref{shuffle_thm2}, it
is enough to prove the case $d=1$.  Let $a=q^n$, $b = a-1$, and $N =
1+(q-2)q^n$. Then, $p^m-a=0$, and the polynomial $f(t)$ of
Theorem~\ref{closed_formula_for_general_q} becomes
\begin{align*}
\sum _{\theta \in \F_q} (t+ \theta)^N 
  = -1- \sum _{0<l<\frac{N}{q-1}} \binom{N}{l(q-1)} t^{N-l(q-1)} \theta^l.
\end{align*}
Taking $a$ as $q^n-1$ in  Proposition~\ref{prop516}, it 
follows that $\binom{N}{l(q-1)}=0$ for
$0<l<{N}/{(q-1)}$. Also, it is easy to see that the result  follows  from
the comparison of the coefficients of $x^j$ in the identity $(1+x)^N =
(1+x)(1+x^{q^n})^{q-2}$, which is valid over $\F_q$ (The author thanks  the
referee for pointing this out). 
\end{example}

\begin{example}
  We continue with Example~\ref{ex2414}. Let us evaluate $S(19,20)$. Let $a
= 20$ and $b = 19$. Then $m = 5$ and $a-b-1 = 0$. Therefore, $f_{k} =
\binom{12}{k}$ for $k = 0,1, \dotsc, 12$. The $k$'s for which $f_{k}\ne 0$ are
$0,4,8$ and 12. Then $S(19,20)= S(20,19) = \set{(1,20), (1,16),
(1,12),(1,8)}$. The
size of $S(19,20)$ is $4 = t_{20}$.
\end{example}

\section{Symmetric closed formulas}

In this section we give a symmetric closed formula for the $f_i$ in
Theorem~\ref{shuffle_thm1}.

\begin{theorem}\label{symmetric_closed_formula_for_general_q}
  Let $q$ be arbitrary. Let $a,b\in \Z_+$; let $m$ be the
smallest integer such that $a +b  \le p^m$. 
Then
\begin{align*}
  \Delta(a,b) = \sum _{i=0}^{b-1} f_i S_1(b-i) + \sum _{j=0}^{a-1} g_j S_1(a-j),
\end{align*}
where $f(t) = H_{a,b}(t) = f_0 + f_{1}t +\dotsb + f_{b-1}t^{b-1}$,  $g(t) =
H_{b,a}(t) = g_0 +  g_1 t +\dotsb + g_{a-1}t^{a-1}$,  $H_{a,b}(t)$ is given
by

\begin{align*}
H_{a,b}(t) = \frac{1}{ t^a } \left(  -  (t^{q-1}-1)^{p^m-a} \sum _{\theta \in
\F_q^*} \left( (t+\theta)^{q-1}-1 \right)^a \bmod t^{a+b} \right),
\end{align*}
and $H_{b,a}$ is obtained interchanging $a$ and $b$. Equivalently,
\begin{align}
H_{a,b}(t) & = 
\sum _{j = 0}^{p^m-a}
\sum _{ k}
\binom{p^m-a}{j}
\binom{a}{k_1,\dotsc,k_{q-1}}
(-1)^{a+j+1}
t^{j(q-1)+ \sigma(k) -a}, \label{Hab1}
\end{align}
where the second sum runs over all $(q-1)$-tuples $k = (k_1,\dotsc,k_{q-1})$ of
non-negative integers such that $k_1 +\dotsb + k_{q-1} = a$ and $\sigma(k)$ is
`even', where $\sigma(k) := \sum _{j=1}^{q-1} j k_{j-1}$ and $j$ and
$\sigma(k)$ are subject to  $j(q-1) + \sigma(k) <a+b$. 
\end{theorem}
\begin{proof}
Following the same steps as in the proof of Theorem~\ref{initial_values_q2}, it
is easy to see that
$\Delta(a,b) =  \sum _{i=0}^{a+b-1} p_i S_1(a+b-i)$,
where  $P(t) = p_0 + p_1t + \dotsb + p_{a+b-1}t^{a+b-1} \in \F_p[t]$ is given by
$$P(t) =  -\frac{[1]^{p^m-a-b}}{t^{p^m-a-b}} [1]^{a+b}
\Delta(a,b) \bmod t^{a+b}.$$ 
Let us write $P(t) = P_a(t) + P_b(t) + P_{a,b}(t)$, 
where $P_{a}(t) = -\frac{[1]^{p^m-a}}{t^{p^m-a}} \sum _{\theta\in \F_q^*}
\frac{[1]^a}{(t+\theta)^a}$ and $P_{a,b}(t)=
-\frac{[1]^{p^m-a-b}}{t^{p^m-a-b}} \sum _{ \theta, \mu \in \F_q^*, \theta\ne
\mu} \frac{[1]^{a}}{(t+\theta)^a}\frac{[1]^b}{(t+\mu)^b}$.
Since  $v_t(P_{a,b}(t)) \ge a+b$, where $v_t(\cdot)$ is the
valuation at $t$, we have 
$P(t) \equiv P_a(t)+ P_b(t)$ modulo $t^{a+b}$. Note that $H_{a,b}(t) =
(1/t^a)(P_a(t)\bmod t^{a+b})$ (observe
that $v_t(P_a(t))\ge a$) and that $H_{b, a}(t)$ is obtained
by interchanging $a$ and $b$.  Suppose for example that 
$\min\{a,b\}=a$. Then $P_{a}(t)+P_{b}(t)$ modulo
$t^{a+b}$ equals 
\begin{align*}
t^a f(t) + t^b g(t)  =
\sum _{i=a}^{b-1}f_{i-a} t^{i} + 
\sum _{i=b}^{a+b-1}(f_{i-a}+g_{i-b})t^{i}.
\end{align*}
Therefore, 
\begin{align*}
\Delta(a,b)  = 
\sum _{i=0}^{b-1} f_{i} S_1(b-i) +
\sum _{i=0}^{a-1} g_{i}S_1(a-i).
\end{align*}
Now we compute $P_{a}(t)$. The factor $-[1]^{p^m-a}/t^{p^m-a} =
-(t^{q-1}-1)^{p^m-a}$  brings in the binomial coefficient. 
Let $\theta \in \F_q^*$. Then, raising $[1]/(t+\theta) = \sum _{i=1}^{q-1}(-t)^i
\theta^{q-1-i}$ to the $a$-th power, by the multinomial theorem, brings in
$\sigma(k)$ and the multinomial coefficient. Summing over $\theta\in
\F_q^*$ and using the fact that 
$\sum_{\xi\in\F_q^*}\xi^\ell$ is $-1$ or $0$, depending on  $\ell$ being `even'
or not, we get 
\begin{align*}
P_{a}(t) = \sum _{j=0}^{p^m-a} \sum _{ \substack { k_1 +\dotsb + k_{q-1}=a \\ 
q-1\mid\sigma(k)}}
\binom{p^m-a}{j} \binom{a}{k_1,\dotsc,k_{q-1}}
(-1)^{a+j+1} t^{j(q-1)+\sigma(k) }.
\end{align*}
Since  $H_{a,b}(t)$ is $1/t^a$ times the remainder of $P_{a}(t)$ when divided
by $t^{a+b}$, we see that $H_{a,b}(t)$ is exactly as claimed.
\end{proof}

When $q=2$,  we get the following compact formula.

\begin{corollary}\label{symmetric_closed_formula_q2}
  Let $q = 2$. Let $a,b \in \Z_+$ and let $m$ be the smallest integer such that
$a+b \le 2^m$. Then
\begin{align*}
  \Delta(a,b) = \sum _{j=0}^{b-1}
\binom{2^m-a}{j}  S_1(b-j) + \sum _{i=0}^{a-1} \binom{2^m-b}{i} S_1(a-i).
\end{align*}
\end{corollary}
\begin{proof}
  For $q=2$, the equation~\eqref{Hab1} becomes $H_{a,b}(t) = \sum
_{j=0}^{b-1}\binom{p^m-a}{j}t^j$.
\end{proof}

{\bf Acknowledgments.}
This work has been developed under
the direction of Dr. Dinesh S. Thakur at the University of Arizona and Dr.
Gabriel D. Villa Salvador at the Departamento de Control Autom\'atico del
Centro de Investigaci\'on y de Estudios Avanzados del IPN (Cinvestav-IPN) in
Mexico City. I thank Javier Diaz-Vargas for his suggestions and advice. I
want
to express my gratitude to the Universidad Aut\'onoma de Yucat\'an and the
Consejo Nacional de Ciencia y Tecnolog\'ia for their financial support. I thank
 the   Sage project~\cite{Sage444} contributors  for providing the
environment to develop this research. I thank the referee for several
helpful suggestions.


\newcommand{\etalchar}[1]{$^{#1}$}

\end{document}